\documentclass{article}
\usepackage{verbatim}
\usepackage{textcomp}

\usepackage{enumerate}
\usepackage{marginnote}
\usepackage{pifont}
\usepackage[normalem]{ulem}
\usepackage[all]{onlyamsmath}
\usepackage{booktabs}
\usepackage{dcolumn}
\usepackage{siunitx}

\usepackage{paralist}
\usepackage{authblk}

\providecommand{\newoperator}[3]{%
  \newcommand*{#1}{\mathop{#2}#3}}

\newoperator{\minimize}{\mathrm{minimize}}{\limits}
\newoperator{\optimize}{\mathrm{optimize}}{\limits}
\newoperator{\subjecto}{\mathrm{subject~to}}{\nolimits}

\newcommand{\units}{modules\xspace}
\newcommand{\chainleft}{\ensuremath{C_{\text{L}}}\xspace}
\newcommand{\errorcontributing}{error-contributing\xspace}
\newcommand{\errorchain}{\ensuremath{\mathrm{EC}}\xspace}

\newcommand{\xycarrychain}[2]{\text{#1-#2}\text{ carry chain}\xspace}

\newcommand{\ijcarrychain}{\textit{i-j}\text{ carry chain}\xspace}
\newcommand{\ijcarrychainone}{\textit{i\ensuremath{_{1}}-j\ensuremath{_{1}}}\text{ carry chain}\xspace}
\newcommand{\ijcarrychaintwo}{\textit{i\ensuremath{_{2}}-j\ensuremath{_{2}}}\text{ carry chain}\xspace}
\newcommand{\pqcarrychain}{\textit{p-q}\text{ carry chain}\xspace}

\newcommand{\chains}{\ensuremath{\mathcal{C}}\xspace}

\newcommand{\nbit}{$n$\nobreakdash-\hspace{0pt}bit\xspace}
\newcommand{\nnbit}{$(n+1)$\nobreakdash-\hspace{0pt}bit\xspace}

\newcommand{\ex}{\mathrm{Er}\xspace}

\newcommand{\exavgroman}{\ensuremath{\mathrm{Er}_{\mathrm{avg}}}\xspace}

\newcommand{\chain}{\ensuremath{C}\xspace}
\newcommand{\fullchain}[2]{\ensuremath{C_{#1 #2}}\xspace}

\newcommand{\veca}{\ensuremath{\boldsymbol{a}}\xspace}
\newcommand{\vecb}{\ensuremath{\boldsymbol{b}}\xspace}

\newcommand{\vecs}{\ensuremath{\boldsymbol{s}}\xspace}

\newcommand{\vecc}{\ensuremath{\boldsymbol{c}}\xspace}

\usepackage[english]{babel}
\usepackage[hypcap]{caption}
\usepackage{colortbl}
\usepackage{algorithm}
\usepackage{algpseudocode}
\usepackage{alltt}
\usepackage{amscd}
\usepackage{amsfonts}
\usepackage{amsmath}
\usepackage{amsthm}
\usepackage{amsxtra}
\usepackage{array}
\usepackage{bigdelim}
\usepackage{blindtext}
\usepackage{braket}
\usepackage{cite}
\usepackage[yyyymmdd,hhmmss]{datetime}

\usepackage{delarray}
\usepackage{empheq}
\usepackage{epstopdf}
\usepackage{esvect}
\usepackage{eucal}
\usepackage{eufrak}
\usepackage{floatflt}
\usepackage{graphicx}
\usepackage{lineno}
\usepackage{makeidx}
\usepackage{mathtools}
\usepackage{mdframed}
\usepackage{microtype}
\usepackage{multirow}
\usepackage{newclude}
\usepackage{relsize}
\usepackage{remreset}
\usepackage{showidx}
\usepackage{tabularx}
\usepackage[textwidth=2in,textsize=scriptsize,backgroundcolor=white,bordercolor=red,linecolor=red]{todonotes}
\usepackage{upref}
\usepackage{xspace}
\usepackage{xparse}
\usepackage{xifthen}
\mathtoolsset{showonlyrefs=true}

\usepackage[pdfauthor=Zvi~M.~Kedem~and~Kirthi~Krishna~Muntimadugu,pdfcreator= ,pdfproducer= , pdftitle=Mathematical~Modeling~of~General~Inaccurate~Adders,colorlinks,allcolors=black,backref=true,bookmarksnumbered=true,bookmarksopen=true,CJKbookmarks=true,hyperfigures=true,verbose=true]{hyperref}

\newcommand{\sgn}{\ensuremath{\mathrm{sgn}}\xspace}
\newcommand{\signhat}{\ensuremath{\widehat{\sgn}}\xspace}

\makeatletter
\providecommand\@dotsep{5}
\def\listoftodos{{\section*{List of Todos}}\@starttoc{tdo}}
\makeatother

\makeatletter
\@removefromreset{footnote}{chapter}
\makeatother

\DeclareDocumentCommand{\Pr}{s m}{% \Pr[*]{..}
  \operatorname{Pr}%
  \IfBooleanTF{#1}% Condition on *
    {#2}% Print only the argument in starred * version
    {\left[#2\right]}% Print bracketed argument [ ] in unstarred version
}%

\providecommand*{\pr}%
{\ensuremath{\mathrm{Pr}}}

\providecommand*{\eu}%
{\ensuremath{\mathrm{e}}}
\providecommand*{\iu}%
{\ensuremath{\mathrm{j}}}

\providecommand{\newoperator}[3]{%
\newcommand*{#1}{\mathop{#2}#3}}

\makeatletter
\providecommand*{\diff}%
{\@ifnextchar^{\DIfF}{\DIfF^{}}}
\def\DIfF^#1{%
\mathop{\mathrm{\mathstrut d}}%
\nolimits^{#1}\gobblespace}
\def\gobblespace{%
\futurelet\diffarg\opspace}
\def\opspace{%
\let\DiffSpace\!%
\ifx\diffarg(%
\let\DiffSpace\relax
\else
\ifx\diffarg[%
\let\DiffSpace\relax
\else
\ifx\diffarg\{%
\let\DiffSpace\relax
\fi\fi\fi\DiffSpace}

\newoperator{\loglike}%
{\mathrm{loglike}}{\nolimits}

\newoperator{\notloglike}%
{\mathrm{notloglike}}{\limits}

\makeatother

\theoremstyle{plain}

\newtheorem{lemma}{Lemma}

\newtheorem{theorem}{Theorem}

\theoremstyle{definition}
\newtheorem{definition}{Definition}
\newtheorem{example}{Example}

\theoremstyle{remark}
\newtheorem{assumption}{Assumption}

\newcommand\Section[2][]{\def\@tempa{#1} \protect\boldmath\ifx\@tempa\@empty\section[#2]{#2}\else \section[#1]{#2}\fi\protect\unboldmath}
\newcommand\Subsection[2][]{\def\@tempa{#1} \protect\boldmath\ifx\@tempa\@empty\subsection[#2]{#2}\else \subsection[#1]{#2}\fi\protect\unboldmath}
\newcommand\Subsubsection[2][]{\def\@tempa{#1} \protect\boldmath\ifx\@tempa\@empty\subsubsection[#2]{#2}\else \subsubsection[#1]{#2}\fi\protect\unboldmath}
\newcommand\Paragraph[2][]{\def\@tempa{#1} \protect\boldmath\ifx\@tempa\@empty\paragraph[#2]{#2}\else \paragraph[#1]{#2}\fi\protect\unboldmath}
\makeatother

\newcommand{\eq}[1]{\hyperref[#1]{Eq.~\eqref{#1}}}
\DeclarePairedDelimiter\abs{\lvert}{\rvert}%
\DeclarePairedDelimiter\bigabs{\bigl\lvert}{\bigr\rvert}%
\DeclarePairedDelimiter\Bigabs{\Bigl\lvert}{\Bigr\rvert}%
\DeclarePairedDelimiter\norm{\lVert}{\rVert}%

\makeatletter
\let\oldabs\abs
\def\abs{\@ifstar{\oldabs}{\oldabs*}}
\let\oldnorm\norm
\def\norm{\@ifstar{\oldnorm}{\oldnorm*}}
\makeatother

\usepackage{fixltx2e}
\newcommand{\SAE}{\ensuremath{\mathrm{SAE}}\xspace}
\newcommand{\oldpaper}{\cite{KMMP:ISPLED2011}\xspace}

\newcommand{\truesign}{\ensuremath{\top}\xspace}
\newcommand{\falsesign}{\ensuremath{\bot}\xspace}

\begin{document}
\title{Mathematical Modeling\\ of General Inaccurate Adders}

\author{Zvi~M.~Kedem\thanks{Department of Computer Science, Courant Institute of Mathematical Sciences, New York University, New York, NY 10012-1180, USA.} \and Kirthi~Krishna~Muntimadugu\thanks{Austin, TX, USA.}}

% line to be shifted left slightly.  This is one of those "LaTeX things".  For
% instance, "\textbf{A} \textbf{B}" will typeset as "A B" not "AB".  To get
% you are supposed to have spaces between the names.  For what it is worth,

\markboth{Journal of \LaTeX\ Class Files,~Vol.~11, No.~4, December~2012}%
{Shell \MakeLowercase{\textit{et al.}}: Bare Demo of IEEEtran.cls for Journals}



\maketitle

\begin{abstract}
  Inaccurate circuits make possible the conservation of limited resources, such as energy.  But effective design of such circuits requires an understanding of resulting tradeoffs between accuracy and design parameters, such as voltages and speed of execution.  Although studies of tradeoffs have been done on specific circuits, the applicability of those studies is narrow.  This paper presents a comprehensive and mathematically rigorous method for analyzing a large class of inaccurate circuits for addition.  Furthermore, it presents new, fast algorithms for the computation of key statistical measures of inaccuracy in such adders, thus helping hardware architects explore the design space with greater confidence.
\end{abstract}

% % Note that keywords are not normally used for peerreview papers.
% \begin{IEEEkeywords}
%   ``not normally used for peerreview papers''
% \end{IEEEkeywords}

% \fi
% creates the second title.  It will be ignored for other modes.
% \IEEEpeerreviewmaketitle
\section{Introduction}
\label{sec:introduction}

Each day, situations arise in which accuracy is traded off for some other objective.  For example, noticeable changes occur in the quality video that is streamed online as the amount of bandwidth that can be allocated changes.   This is an instance of software-controlled, inaccurate computation. A similar notion can be applied to hardware as well, where the energy that can be allocated to computations is constrained. In this paper, we study inaccurate circuits, specifically circuits for binary additions that may produce incorrect sums for some inputs. Such incorrect sums may be the result of the outputs being read before the computation had a chance to complete (such as approximate adders which violate the critical path, studied e.g. by Kedem et al.~\oldpaper), or they may be the result of the hardware not being designed to be ``fully'' correct (such as pruned adders studied e.g. by Lingamneni et al.~\cite{AvinashENPP11}). In brief, such circuits are designed and operated in a way that trades off accuracy for the usage of a limited resource such as  energy. Particularly because of energy limitations, such circuits are the subject of rapidly growing interest.

An inaccurate ripple-carry adder (RCA) and the expected error under various assignments of voltages to its gates were studied in~\oldpaper, and we refer interested readers to that paper and its references for motivation, additional background, and related work, beyond what is described here and what is required for understanding the results of this paper.  To emphasize that we are considering inaccurate circuits for addition, we will refer to such circuits as \emph{pseudo-adders}.  Of course, the set of correct adders is a subset of the set of pseudo-adders.

The focus of the paper is exclusively on the mathematical properties of the outputs produced by pseudo-adders.  The mathematical properties we study are fundamental to the analysis of important factors in the design of pseudo-adders, such as delays, layouts, voltage assignments, and the energy consumed during the operation.

In order to meaningfully  explore the design space, it is necessary to have metrics for the evaluation of alternative designs.  One such metric would be, of course, the inherent inaccuracy of a pseudo-adder.  A useful measure for this is the expected error, and~\oldpaper developed a fast way for computing it for one type of overclocked RCA.  That paper, however, had limited applicability to other adders because the technique that it developed was targeted to RCAs and crucially relied on specific characteristics of RCAs.

Results presented here go much beyond that paper.  First, we present a general method capable of handling a large class of pseudo-adders. As an example, we show the applicability of our method to the highly parallel Kogge-Stone adder.  The set of statistical measure is also richer.  We provide not only an algorithm for computing the expected error but also algorithms for computing the mean squared error and the maximum error. All of the algorithms are very fast.

To compute the statistical measures listed above, ostensibly one needs to consider the results of computing $a+b$ for all possible pairs of inputs $(a,b)$.  As there are $2^{2n}$ of them for an $n$-bit adder, such an approach would be impossible for other than rather modest values of $n$.  One could optimize somewhat by for example only considering pairs for which $a\leq b $, but savings of this type are not significant.

As was done in~\oldpaper, we reduce the computational effort required by only analyzing the errors in carry chains, thus, in effect having to deal with only $n(n+1)/2$ pairs of inputs while still computing the statistical measures valid for the complete set of input pairs.  In consequence, the number of input pairs that need to be considered is reduced dramatically.  For $n=64$, instead of considering $340.282\times 10^{36}$ pairs, we need to consider only $2.08\times 10^{3}$ pairs, a reduction by 35 orders of magnitude.

For a general background on computer arithmetic, see e.g., Parhami~\cite{Parhami2010}.  To facilitate comparison with~\oldpaper by interested readers, we will use mainly its notation, but we will modify it where it makes the presentation of our results clearer.

The rest of the presentation is organized as follows. We first briefly review the notion of carry chains and formulate our model and the class of pseudo-adders we will analyze.   We then review the results obtained in~\oldpaper, also showing when they are \emph{not} applicable. We then proceed to our more general treatment, formulating and proving the characteristics of the errors that can occur in them. This leads to the formulation of fast algorithms for the computation of the statistical measures of interests. We show an example of an interaction of errors for a pair of inputs in a Kogge-Stone-based pseudo-adder. We conclude by recapitulating our results.

\section{Inputs and carry chains}
\label{sec:inputs-carry-chains}

The fundamental issue in the design of fast adders is the possible presence of carry chains.  Carry chains were studied as early as the 19th century by Babbage~\cite{Babbage1864} for decimal adders and in 1946 by Burks et al.~\cite{vonNeumann1946} for binary adders.

We will consider the addition of two numbers $a$ and $b$ from the range $[0,\dotsc,2^{n}-1]$, for some fixed $n$.  Using standard binary representation, such numbers are in one-to-one correspondence with $n$-bit vectors.  It will, however, be more convenient to consider input vectors of length $n+1$: $\veca = \langle a_{n} a_{n-1}\dotso a_{0}\rangle $ and $\vecb = \langle b_{n} b_{n-1}\dotso b_{0}\rangle$, with $a_{n}=b_{n}=0$.  Given two numbers $a$ and $b$,  two sequences \vecc and \vecs of length $n+1$ are  defined by
\begin{align}
  \label{eq:7}
  &c_{k} =a_{k-1}b_{k-1} \lor a_{k-1} c_{k-1} \lor b_{k-1}c_{k-1}
 \\
\label{eq:72}&s_{k} =a_{k} \oplus b_{k} \oplus c_{k}
\end{align}
for $k=0,1,\dotsc,n$, with $c_{0}=0$.  Thus, \vecs corresponds to the sum $a+b$ and $c_{k}$ is the carry needed to compute $s_{k}$, for $k=1,\dotsc,n$.  Of course, $c_{k}$ is a function of $\langle a_{k-1}a_{k-2} \dotso a_{0} \rangle$ and $\langle b_{k-1}b_{k-2} \dotso b_{0} \rangle$.  In the rest of the paper, when $a$ and $b$ are understood from the context, we may omit them.

We \emph{adapt} a standard definition of a carry chain:
\begin{definition}
  \label{def:newcarrychains}
Given \nbit numbers  $a$ and $b$, there is a \emph{carry chain starting} at $i$ and \emph{ending} at $j$ for $1 \leq i \leq j \leq n$ when
\begin{enumerate}
    \item $a_{i-1}=b_{i-1}=1$
    \item $a_{k} \not = b_{k}$ for $i \leq k \leq j-1$
    \item $a_{j} = b_{j}$.
\end{enumerate}
We will refer to such a carry chain \emph{\ijcarrychain} and also say that it is \emph{from} $i$ \emph{to} $j$.  We will say that the pair $(a,b)$ \emph{generates} this carry chain.  The set of all the carry chains generated by $(a,b)$ will be denoted by $\chains(a,b)$.

Note that our terminology for carry chains is slightly different from the more common one in which a carry chain is a sequence of positions that generate and propagate a carry.  That is, using our notation, in the common definition the carry chain starts at position $i-1$ and it ends in position $j-1$.  We chose our notation, as it is more convenient to focus \emph{not} on the positions from which there are \emph{outgoing} carries of $1$, but on the positions into which there are \emph{incoming} carries of $1$ because those are the positions where errors in the sum may occur.
\end{definition}
Fixing $(a,b)$ and using the notation
\begin{equation}
  \label{eq:13}
\fullchain{i}{j}(a,b) =
\begin{dcases*}
 \truesign &  if there is a carry chain from $i$  to $j$
\\
        \falsesign & otherwise,
    \end{dcases*}
  \end{equation}
  two well known and easily proved~\oldpaper properties are
  \begin{lemma}
    \label{lem:3}
  If $\fullchain{i}{j}(a,b) = \fullchain{p}{q}(a,b) = \truesign$ and $(i,j) \not = (p,q)$, then the intervals $[i,j]$ and $[p,q]$ are \emph{disjoint}, or equivalently, the two carry chains \emph{do not overlap}.
  \end{lemma}
  \begin{lemma}
    \label{lem:4}
  If $\fullchain{i}{j}(a,b)=\truesign$ and $s=a+b$, then $s_{k}=0$, for $k= i,\dotsc,j-1$, and $s_{j}=1$.
  \end{lemma}

\section{A model for pseudo-adders}
\label{sec:model-pseudo-adders}

We will model circuits under consideration by directed acyclic graphs (DAGs) with $2n+1$ input (source) vertices and $n+1$ output (sink) vertices.  The two inputs $\boldsymbol{a}$ and $\boldsymbol{b}$ and the input carry $c_{0}$ are supplied at the input vertices at time $t=0$ and the sum $\boldsymbol{s}$ is supposed to be read from the output vertices at some later time $T$.  However, at time $t=T$, the output vertices do not necessarily contain the correct sum $\boldsymbol{s}$ but some vector $\boldsymbol{s'}$, which may differ from $\boldsymbol{s}$, and in general the correct \vecs may not ever be available at the outputs.

We can write an equation analogous to \eqref{eq:72}:
\begin{equation}
  \label{eq:18}
  s_{k}' =a_{k} \oplus b_{k} \oplus c_{k}',
\end{equation}
where $c_{k}'$ for $k=1,2, \dotsc, n$ is a function of $\langle a_{k-1}a_{k-2} \dotso a_{0} \rangle$ and $\langle b_{k-1}b_{k-2} \dotso b_{0} \rangle$, just as $c_{k}$ was.  In the completely general setting, no assumptions can be made about the relation between  $c_{k}$ and $c_{k}'$, so interesting results are not likely to be possible.  We thus restrict our treatment to a useful subset of cases.

\begin{definition}
  \label{def:2}
  A pseudo-adder will be called \emph{conservative} if for all $k$, $c_{k}' \leq c_{k}$, that is whenever $c_{k}=0$ then also $c_{k}'=0$.  Therefore, there will be no spurious carries.
\end{definition}

A situation for which such pseudo-adders appear naturally is one in which the circuit itself is a correct adder, but it is overclocked.  Then, it may be the case that for some $k$ the circuit's ``default'' value for $c_{k}$ is $0$, but for some inputs $c_{k}=1$ and that information is not properly processed by time $t=T$.  In consequence, $c_{k}'=0$ and $s_{k}' \not = s_{k}$.  However, there are conservative pseudo-adders that are not simply an overclocked version of correct adders but are incorrect by their architecture.

We will restrict ourselves to conservative pseudo-adders.  We  make two technical  assumptions, which may seem obvious but they need to be explicitly stated in order to define the error contributed by a carry chain.

\begin{assumption}
  \label{assump:2}
    Because we are dealing with conservative pseudo-adders,  errors can only occur  when for some $k$, $c_{k}=1$ but $c_{k}'=0$.  Therefore, $i \leq k \leq j$, for some \ijcarrychain.,  As $c_{k}=1$ is determined independently of the values of $\langle a_{i-2},a_{i-3} \dotso a_{0} \rangle$ and $\langle b_{i-2}b_{i-3} \dotso b_{0} \rangle$, we assume that $c_{k}'=0$ is also determined independently of those values.  So we assume that the input values in the positions \emph{lower} than $i-1$ cannot influence what happens during the propagation of the carry generated \emph{at} $i-1$.
\end{assumption}

\begin{assumption}
  \label{assump:1}
 Of course the addition $a+b$ is commutative but one can envision conservative pseudo-adder circuits in which the behavior of the computation is not commutative.  Imagine a correct adder design which however has a fabrication error: it ignores the supplied value of $a_{0}$ and always propagates it as $0$.  In such a circuit, $0+1$ will result in $1$ but $1+0$ will result in $0$.  Since we are interested in the properties of a carry chain that are independent of the specific input pair generating it, we will assume that $\boldsymbol{c'}$ is the same for $a+b$ and $b+a$.  \label{bb:1}
\end{assumption}

Consider any carry chain \chain and a pair of inputs $(a,b)$ that generates it.  Define a pair of inputs $(a^{C},b^{C})$ as follows: $a_{0}^{C}=a_{1}^{C}= \dotsm = a_{i-2}^{C}=0$, $a_{i-1}^{C}=a_{i-1}$, $a_{i}^{C}=a_{i}$, {\ldots}, $a_{j-1}^{C}=a_{j-1}$, $a_{j}^{C}=a_{j+1}^{C}= \dotsm = a_{n-1}^{C}=0$, and analogously for $b$.  In effect we have isolate the carry chain by ``zeroing out'' positions irrelevant to its appearance.

We can now state
\begin{definition}
  \label{def:3}
  Let \chain be a carry chain and $(a,b)$ any pair of inputs generating \chain.  \emph{For the associated $(a^{\chain},b^{\chain})$} let $s$ and $s'$ denote respectively the true sum and the sum computed by the pseudo-adder.  Then we define the error of the carry chain as
  \begin{equation}
    \label{eq:21}
    \errorchain(\chain) = s -s'
  \end{equation}
(although more conventionally one might use $s'-s$).
\end{definition}
Note the importance of the two assumptions above.  They allowed us to define the error of the carry chain by considering \emph{any pair} of inputs generating the chain for defining the chain's error.  Our being able to do that made the concept of the error of the carry chain meaningful.

The implications for further development can be stated succinctly:
\begin{lemma}
  \label{lem:5}
  In a conservative pseudo-adder errors are determined by carry chains.
\end{lemma}
\begin{proof}
  The lemma holds because if two sums $a^{(1)}+b^{{(1)}}$ and $a^{(2)}+b^{{(2)}}$ have carry chains in the same positions, then the \emph{computed} sums for them are the same.
\end{proof}

\section{Reducing from errors in sums to errors in carry chains}
\label{sec:reducing-errors-sums}

\subsection{Errors and carry chains}
\label{sec:errors-carry-chains}

The material in this section contains mostly a general reformulation of the material previously presented  in~\oldpaper for a more restrictive case, but we start with an example of the case that could not be treated there.

\begin{example}
  \label{ex:1}
  \begin{figure}
\centering
% \begin{tabular}{@{\hskip 18pt}l@{\hskip 24pt}l@{\hskip 18pt}}
\begin{tabular}{r@{\hskip 5em}rrrrrrrrr@{\hskip 0.5em}r@{\hskip 0em}rr}
\toprule
% \multicolumn{1}{c}{\textbf{Position}}&8&7&6&5&4&3&2&1&0\\
&\multicolumn{8}{c}{\hspace{1em}\textbf{Position}}\\
&&7&6&5&4&3&2&1&0\\
\noalign{\vskip 0.5ex}
\midrule
\noalign{\vskip 0.5ex}
\multicolumn{1}{l}{$\boldsymbol{a}\hspace{1em}\longrightarrow\hspace{1em}$}&$0$&$0$&$1$&$0$&$1$&$0$&$1$&$1$&$0$\\
\multicolumn{1}{l}{$\boldsymbol{b}\hspace{1em}\longrightarrow\hspace{1em}$}&$0$&$0$&$0$&$1$&$1$&$1$&$0$&$1$&$1$\\
\multicolumn{1}{l}{$\boldsymbol{c}\hspace{1em}\longrightarrow\hspace{1em}$}&$0$&$1$&$1$&$1$&$1$&$1$&$1$&$0$&$0$\\
\multicolumn{1}{l}{$\boldsymbol{s}\hspace{1em}\longrightarrow\hspace{1em}$}&$0$&\cellcolor{orange}$1$&\cellcolor{orange}$0$&\cellcolor{orange}$0$&\cellcolor{yellow}$1$&\cellcolor{yellow}$0$&\cellcolor{yellow}$0$&$0$&$1$\\
\noalign{\vskip 0.5ex}
\midrule
\noalign{\vskip 0.5ex}
\multicolumn{1}{l}{$\boldsymbol{a}\hspace{1em}\longrightarrow\hspace{1em}$}&$0$&$0$&$1$&$0$&$1$&$0$&$1$&$1$&$0$\\
\multicolumn{1}{l}{$\boldsymbol{b}\hspace{1em}\longrightarrow\hspace{1em}$}&$0$&$0$&$0$&$1$&$1$&$1$&$0$&$1$&$1$\\
\multicolumn{1}{l}{$\boldsymbol{c'}\hspace{1em}\longrightarrow\hspace{1em}$}&$0$&$1$&$0$&$0$&$0$&$1$&$1$&$0$&$0$\\
\multicolumn{1}{l}{$\boldsymbol{s'}\hspace{1em}\longrightarrow\hspace{1em}$}&$0$&$\cellcolor{orange}1$&\cellcolor{orange}$1$&$\cellcolor{orange}1$&$\cellcolor{yellow}0$&\cellcolor{yellow}$0$&\cellcolor{yellow}$0$&$0$&$1$\\
\noalign{\vskip 1ex}
\bottomrule
\end{tabular}
\caption{Carry and sum bits for an  addition in an $8$-bit adder.  The upper block shows correct addition and the lower block an erroneous one.  Carry chains are highlighted.}
\label{table:example-carries}
\end{figure}

We consider the computation of $a+b$, where $\boldsymbol{a}= \langle 001010110\rangle $ and $\boldsymbol{b}=\langle 000111010\rangle $, generating the \xycarrychain{2}{4} and the \xycarrychain{5}{7}.  The situation is depicted in Fig.~\ref{table:example-carries}.  The top block illustrates the correct computation and the bottom block an incorrect one, which can take place in a pseudo-adder.  The various values are $\vecc = \langle011111100 \rangle$, $\boldsymbol{s}=\langle 010010000\rangle $, $\boldsymbol{c'} = \langle 010001100\rangle$, and $\boldsymbol{s'}= \langle 011100001\rangle$.  Note that as $c'_{k} \leq c_{k}$ for every $k$, this is a conservative pseudo-adder (at least for these inputs).  Note also that the errors contributed by the two carry chains are respectively $-96$ and $+16$, so the errors contributed by carry chains can be both negative and positive.   In Section~\ref{sec:example:-kogge-stone}, we will see exactly how the situation in the lower block can occur in an
 KSA-based pseudo-adder.  This completes the example.
\end{example}

 The algorithm in~\oldpaper was limited in \emph{not accounting} for the possibility of negative errors in carry chains; it considered only overclocked RCAs, where such errors cannot occur, as we discuss further in Section~\ref{sec:case-ripple-carry}.

Returning to a general conservative pseudo-adder,  we examine more closely the difference between $s_{k}$ and $s_{k}'$ for some fixed $k$.  We need to consider several cases:
\begin{enumerate}%[list:carries]
    \item
  If $c_{k} \not = c_{k}' $, then $c_{k}=1$, $c_{k}' =0$, and $k$ is such that for some $(i,j)$, $i \leq k \leq j$, we have $\fullchain{i}{j}(a,b)=\top$.  Then there are two cases: \label{aa:1}
  \begin{enumerate}
      \item
    If $k < j$, then $s_{k}=0$, $s_{k}'=1$, and $s_{k}-s_{k}' = -1$.  \label{aa:1a}
      \item
    If $k = j$, then $s_{k}=1$, $s_{k}'=0$, and $s_{k}-s_{k}' = +1$.  \label{aa:1b}
  \end{enumerate}
    \item
  If $c_{k}= c_{k}'$ then $s_{k}-s_{k}'=0$.  \label{aa:2}
\end{enumerate}

We are interested in statistical properties of errors in the computations by pseudo-adders.  A natural  measure to consider is the absolute value of the average error.

The absolute error in the addition $a+b$ is
\begin{equation}
  \label{eq:16}
  \ex(a, b) =  \bigabs {s(a,b)-s'(a,b)}
\end{equation}
and the expectation of the absolute value of the error is
\begin{equation}
  \label{eq:19}
\exavgroman  = \frac{1}{2^{2n}} \smashoperator[r]{\sum_{a,b}} \; \ex( a, b),
\end{equation}
assuming the uniform distribution as we do here.  So we need to compute \emph{the sum of the absolute values of the errors} (\SAE), that is
\begin{equation}
  \label{eq:20}
\SAE =  \smashoperator[r]{\sum_{a,b}}  \ex( a, b),
\end{equation}
a sum of $2^{2n}$ terms, which is an impractical task for other than relatively modest values of $n$.

We will replace \eqref{eq:20} by a certain sum involving errors in the $n(n+1)/2$ carry chains possible in the addition of \nbit numbers and we will use \chain as a variable varying over this set of all the carry chains.  If \chain contributes an error in the sum of a pair of inputs, we refer to \chain as an \emph{\errorcontributing chain} for those inputs.

Because the carry chains generated by a pair of inputs do not overlap, we can rewrite \eqref{eq:20} as
\begin{equation}
  \label{eq:31}
 \SAE = \sum_{a,b}\,  \Bigabs {
  \smashoperator[r]{
      \sum_{\chain \in \chains(a,b)}
}\,
      \errorchain(\chain)  },
\end{equation}
where $\chains(a,b)$ is the set of the carry chains generated by $(a,b)$.

\subsection{The special case of the fine-grained overclocked RCA}
\label{sec:case-ripple-carry}

% We now turn our attention to pseudo-adder obtained from the standard RCA by \emph{fine-grained overclocking}. the modules of the RCA (at whatever level of granularity is technologically feasible) are assigned delays. Even identical modules, such as full adders, may have different delays in different parts of the circuit. In consequence, the adder may not always operated correctly, and the delays may be ``distributed'' to optimize some accuracy measure.

We now turn our attention to a pseudo-adder obtained from the standard RCA by \emph{fine-grained overclocking}. The modules of the RCA (at whatever level of granularity is technologically feasible) are supplied with different voltages. Even identical modules, such as full adders, may have different voltages supplied in different parts of the circuit. In consequence, identical modules in different parts of the adder might have different delays. If operated at a fixed frequency, such an adder may not always operate correctly since sometimes the output might be read faster than the critical path delay. These voltages, and hence the delays, may be ``distributed'' to optimize some accuracy measure as shown in~\oldpaper.

In specific pseudo adders studied in this Section~\ref{sec:case-ripple-carry}, errors in carry chains have a property that greatly simplifies their treatment.  Intuitively, because for each carry chain, the carry ``ripples'' through sequentially, the error contributed by a carry chain can only result in an output sum that is smaller than the true one.  Specifically,
\begin{theorem}
  \label{thm:3}
  In a pseudo-adder based on the ripple-carry adder
  \begin{equation}
  \label{eq:17}
\errorchain(\chain) \geq 0
\end{equation}
for every carry chain  \chain.
\end{theorem}
\begin{proof}
  Consider some \ijcarrychain \chain for which $\errorchain(\chain) \not = 0$.  The correct values are $s_{i}=s_{i+1}= \dotsm = s_{j-1}=0$ and $s_{j}=1$. Due to the structure of the RCA, $\errorchain(\chain) \not = 0$ because the carry has not yet propagated to positions $k+1, k+2, \dotsc, j$, for some $k+1 \leq j$.  Then $s_{i}'=s_{i+1}'= \dotsm = s_{k}'=0$, $s_{k+1}'=\dotsm = s_{j-1}' = 1$, and $s_{j}'=0$. Therefore, $s-s' = 2^{j} - \sum_{\ell=k+1}^{j-1} 2^{\ell} > 0$.
\end{proof}
 In consequence, of course, $\abs{\,\errorchain(\chain)\,} = \errorchain(\chain) $ and therefore
\begin{equation}
  \label{eq:2}
  \sum_{a,b} \, \Bigabs {
\smashoperator[r] {
\sum_{\chain \in \chains(a,b)}
}
\errorchain(\chain)  }   =  \sum_{a,b}
\smashoperator[r]{
  \sum_{\chain \in \chains(a,b)}
}\,
\errorchain(\chain).
\end{equation}
Let $\nu(\chain)$ denote the number of input pairs $(a,b)$ for which \chain appears in $a+b$.  Then we can change the order of summation and therefore
\begin{equation}
  \label{eq:4}
  \sum_{a,b}
\smashoperator[r]{
  \sum_{\chain \in \chains(a,b)}
}\,
\errorchain(\chain)
=  \sum_{\chain}   \errorchain(\chain) \cdot \nu(\chain).
\end{equation}

Crucially relying on \eqref{eq:17} in studying a specific type of pseudo-adder based on the RCA,~\oldpaper used
\begin{equation}
  \label{eq:6}
\SAE
=
\sum_{\chain}  \errorchain(\chain) \cdot \nu(\chain)
\end{equation}
to compute the absolute value of the average error (although actually the result of interest was phrased there equivalently using probabilities of the errors and not directly using the number of input pairs that generate them).  However, for circuits other than RCAs, \eqref{eq:17} \emph{does not necessarily hold} and in such cases \eqref{eq:6} \emph{will not produce the correct  \SAE}.

\subsection{The case of general conservative pseudo-adders}
\label{sec:case-gener-cons}

To treat general pseudo-adders, a different procedure for computing \eqref{eq:31} must be established, which we do next.  Later we apply it in the example of the KSA for which, indeed \eqref{eq:17} does not hold.

The goal is to dispense with absolute values.  As the development is somewhat technical, we start with an example to provide intuitive understanding behind~\eqref{eq:1}.

\begin{example}
  \label{ex:3}
  The example is not realistic but is well suited for providing intuition.  Assume that we pick some carry chain \chain and we want to know what the contribution of the error in \chain to the sum of absolute errors is.  Its error is $\errorchain(\chain)=-6$, and \chain is generated by only three input pairs each of which generates exactly one additional and distinct error-contributing chain.  Those additional chains $C_{1}$, $C_{2}$, and $C_{3}$ contribute respectively errors of $+13$, $-14$, and $+15$.  Then the total errors for those inputs are $+7$, $-20$, and $+9$.  As the absolute values of these errors are larger than that of $\errorchain(\chain)$, the signs of these errors determine the signs of the total error for their respective input pairs. Hence, we will designate these three carry chains as dominating ones for their respective input pairs.

  We now note that we can obtain the absolute value of the sum of errors for some input by taking the error of the individual chains, multiplying each by the signum of the error of the dominating chain, and adding these terms together.  Indeed, the three signa are $+1$, $-1$, and $+1$, and therefore
  \begin{align}
    \label{eq:9}
    &\abs{-6+13} = (+1)(-6)+(+1)(+13)=+7\\
    &\abs{-6-14} = (-1)(-6)+(-1)(-14)=+20\\
    &\abs{-6+15} = (+1)(-6)+(+1)(+15)=+9.
  \end{align}
  We now compute the contribution of $\errorchain(\chain)$ to \SAE.  It is
  \begin{equation}
    \label{eq:15}
    (+1)(-6)+(-1)(-6)+(+1)(-6) = (-6)(+2-1),
  \end{equation}
  and the SAE is
  \begin{equation}
    \label{eq:23}
    (+1)(+13)+(-1)(-14)+(+1)(+15)+(-6)(+2-1).
  \end{equation}
  Note that $2$ was the number of the input pairs generating \chain in which the dominating chain was positive and $1$ was the number of the input pairs generating \chain in which the dominating chain was negative.  The numbers $+2$ and $-1$ on the right-hand side of~\eqref{eq:15} correspond to $\nu^{+}(\chain)$ and $\nu^{-}(\chain)$ in ~\eqref{eq:1}.  This completes the example.
\end{example}

We now proceed to the general case.  Assume that $\chains(a,b)$ has at least one error-contributing carry chain for input pair $(a,b)$.  We take the \emph{leftmost} (highest position numbers) chain contributing an error, refer to it as the \emph{dominating carry chain} for $(a,b)$, and denote it by $\chainleft(a,b)$.  The following lemma is pivotal to the subsequent development as it states that the sign of the total error for a pair of inputs is determined by the leftmost chain that pair generates.

\begin{lemma}
  \label{lem:1}
\begin{equation}
  \label{eq:22}
\sgn \Biggl(\,
    \smashoperator[r]{
      \sum_{\chain \in \chains(a,b)}
      }\,
\errorchain(\chain) \Biggr)
   =
\sgn \Bigl(\errorchain \bigl( C_{\mathrm{L}}(a,b) \bigr) \Bigr),
\end{equation}
where, as usual, $\sgn$ denotes the signum function.
\end{lemma}
% For proof please see Appendix~\ref{sec:proof-lemma-reflem:1}.
\begin{proof}
    The idea behind the proof is that if $k$ is the largest position in which there is an error, it is also the largest position in which there is an error in the positions of the dominating carry chain.

  Let $\boldsymbol{u}$ and $\boldsymbol{v}$ be \nnbit vectors such that for some $k$, $0 \leq k \leq n$, $u_{n} =v_{n}$, $u_{n-1} =v_{n-1}$, \ldots, $u_{k+1}=v_{k+1}$ but $u_{k}=1$ and $v_{k}=0$.  Then, obviously, $u>v$.

  Consider now inputs $a$ and $b$ and the corresponding $s$ and $s'$ such that $s \not = s'$.  Let $k$ be the leftmost position in which $s_{k}\not = s_{k}'$.  Then $\chainleft(a,b)$ is an \ijcarrychain such that $i \leq k \leq j$.  We will write $\chainleft$ for $\chainleft(a,b)$ in this proof.  Recalling a notation from Definition~\ref{def:3}, let $s(\chainleft)$ and $s'(\chainleft)$ denote respectively the correct and the computed sum for the input pair $(a^{\chainleft},b^{\chainleft})$.  Then,  $s(\chainleft) - s'(\chainleft) = \errorchain(\chainleft) $.

  Position $k$ is also the leftmost position in which $\boldsymbol{s}(\chainleft)$ and $\boldsymbol{s'}(\chainleft)$ are different, and furthermore $s_{k}(\chainleft) = s_{k}$ and $s_{k}'(\chainleft) = s_{k}'$.  Therefore $s>s'$ if and only if $s(\chainleft)>s'(\chainleft)$, and as $s-s' = \sum_{\chain \in \chains(a,b)} \errorchain(\chain)$, the lemma is proved.
\end{proof}

As we will see next, the lemma  allows us to replace the absolute sum in \eqref{eq:31} by an appropriate collection of signed (multiplied by $+1$ or $-1$) $\errorchain(\chain)$'s.  Let
\begin{comment}
\begin{equation}
    \signhat(a,b)=
  \ccases{
    0&\text{if }\chains(a,b) = \emptyset   \\
     \sgn \Big(\errorchain \big(\chainleft(a,b) \big) \Big)&\text{otherwise}.
  }
  \label{eq:32}
\end{equation}
\end{comment}
\begin{equation}
    \signhat(a,b)=
  \begin{dcases*}
    0& if $\chains(a,b) = \emptyset   $\\
     \sgn \Big(\errorchain \big(\chainleft(a,b) \big) \Big)&otherwise.
   \end{dcases*}
   \label{eq:32}
\end{equation}
Recall that the absolute value of the error contributed by \chain to the computation of $a+b$ is the same for all pairs of inputs $(a,b)$ which generate \chain.  Let
\begin{equation}
  \label{eq:3}
 \nu^{\sigma}(\chain) = \Bigl\lvert \big \{  (a,b) \mid C \in \chains(a,b) \wedge  \signhat(a,b) = \sigma  \big \}  \Bigr\rvert,
\end{equation}
where the outer bars denote the cardinality of the set.  (Note that $\nu^{\sigma}$ is \emph{not} $\nu$ to the power of $\sigma$, but rather either $\nu^{+}$ or $\nu^{-}$.)  The number $\nu^{\sigma}(C)$ is the number of input pairs that generate \chain and in which the sign of the dominating chain is $\sigma$.

We are now ready to handle the absolute values in~\eqref{eq:31} using $\signhat(a,b)$.  The key property follows.  Consider some pair of inputs $(a,b)$ for which there is an error in the pseudo-adder.  By Lemma~\ref{lem:1}, the sign of the total error in $a+b$ is the sign of the corresponding dominating chain, $\chainleft(a,b)$.  Therefore, to obtain the absolute value of the total error, we can multiply the total error by the sign of the dominating chain, that is by $\signhat(a,b)$.  This leads to the following.
\begin{theorem}
  \label{thm:2}
  \begin{equation}
    \label{eq:1}
    \exavgroman = \frac{1}{2^{2n}} \sum_{C}   \errorchain(\chain) \cdot \big ( \nu^{+}(\chain)   -\nu^{-}(\chain) \big).
  \end{equation}
\end{theorem}
% For proof please see Appendix~\ref{sec:proof-theor-refthm:2}.
\begin{proof}
  We first note that
  \begin{align}
    \label{eq:12}
       \Bigabs {
  \smashoperator[r]{
      \sum_{\chain \in \chains(a,b)}
    }  \errorchain(\chain)
    }
    &= \signhat(a,b) \cdot
\smashoperator[lr] {
  \sum_{\chain \in \chains(a,b)}
}
\errorchain(\chain  )\\
    &=
\smashoperator[lr] {
  \sum_{\chain \in \chains(a,b)}
}
\signhat(a,b) \cdot \errorchain(\chain  ).
  \end{align}

Then, having dispensed with the absolute values, we have
  \begin{align}
    \label{eq:8}
 \sum_{a,b} \, \Bigabs {
  \smashoperator[r]{
      \sum_{\chain \in \chains(a,b)}
}\,
\errorchain(\chain)  } & =
\sum_{a,b} \,
\smashoperator[r]{
  \sum_{\chain \in \chains(a,b)}
  }\,
  \signhat(a,b) \cdot \errorchain(\chain)  \\
   & =
\sum_{C}
\smashoperator[r]{
  \sum_{\{(a,b)\, \mid \, \chain \in \chains(a,b)\}}
  }\,
\signhat(a,b) \cdot \errorchain(\chain)  \\
&=\sum_{C}   \errorchain(\chain) \cdot \big ( \nu^{+}(\chain)   -\nu^{-}(\chain) \big),
\end{align}
and the theorem immediately follows.
\end{proof}

In the simple case of the RCA, $\nu^{-}(\chain)=0$ and we can write $\nu$ for $\nu^{+}$, simplifying the equation and  getting~\eqref{eq:6}.

To compute \eqref{eq:19} fast, it is enough to compute all the terms in \eqref{eq:1} fast.  For every \chain, one needs to compute $\errorchain(\chain)$, and we return to this later, assuming here that we have these values for the pseudo-adder under consideration.

\begin{figure}
\centering
\begin{tabular}{@{\hskip 18pt}l@{\hskip 24pt}l@{\hskip 18pt}}
\toprule
\textbf{Positions}&\textbf{Inputs}\\
\midrule
$0,\dotsc,i-2$&$00$, $01$, $10$, $11$\\
$i-1$&$11$\\
$i,\dotsc,j-1$&$01$, $10$\\
$j$&$00$, $11$\\
$j+1,\dotsc,p-2$&$00$, $01$, $10$, $11$\\
$p-1$&$11$\\
$p,\dotsc,q-1$&$01$, $10$\\
$q$&$00$\\
$q+1,\dotsc,n-1$&$00$, $01$, $10$\\
\bottomrule
\end{tabular}
\caption{Input conditions for  \ijcarrychain and for the corresponding leftmost \pqcarrychain.  The left column lists the positions and the right column lists the pairs of inputs allowed in the positions.  Note that possibly $q=n$ and then the last row does not apply.  The relation between a carry chain and the probability of an input pair producing it was already studied in~\cite{vonNeumann1946}.}
\label{table:carry-chains-conditions}
\end{figure}

\subsection{Computing fast the total effect of a carry chain error}
\label{sec:comp-effect-carry}

We need to compute $\nu^{+}(\chain)$ and $\nu^{-}(\chain)$ where \chain is an error-contributing \ijcarrychain and $1 \leq i \leq j \leq n$.  If \chain is the leftmost error-contributing chain generated by some input pair $(a,b)$, then $\signhat(a,b) = \sgn\big(\errorchain(\chain) \big)$.  This is the simpler case, so we will consider here the case where \chain is not the leftmost error-contributing chain and let \pqcarrychain be the leftmost error-contributing chain generated by $(a,b)$.  For this situation, $n \geq q \geq p > j \geq i \geq 1$ and the conditions in Fig.~\ref{table:carry-chains-conditions} hold.  The number of input configurations producing such a pair of carry chains is
\begin{enumerate}
    \item
  If $n=q$ then  $2^{n-p}2^{j-i}4^{(p-1)-(j-i+1)}$.
    \item
  If $n>q$ then  $3^{n-q+1}2^{q-p}2^{j-i}4^{(p-1)-(j-i+1)}$.
\end{enumerate}

Keeping $i$ and $j$ fixed, we consider all the possible values for $p$ and $q$.  For some inputs, the \ijcarrychain is the leftmost error-contributing chain, and for some inputs there is another error-contributing chain, \pqcarrychain, and the latter is the leftmost such carry chain.  By examining all the cases, we can compute $\nu^{+}(\chain)$ and $\nu^{-}(\chain)$.  Varying the \ijcarrychain over $1 \leq i \leq j \leq n$ we treat all the carry chains and from this we compute the needed $ \exavgroman$, using \eqref{eq:19}.

It is immediately clear how to compute  \eqref{eq:1} in $\Theta(n^{4})$   steps by considering all relevant quadruples  $(i,j,p,q)$.  We sketch next how to do that in $\Theta(n^{2})$ steps.

Fix now an error-contributing \ijcarrychain \chain, where $1 \leq i \leq j \leq n$.  Of course we know the number of input pairs  that generate it.  To compute the needed $\nu^{+}$ and $\nu^{-}$ we need to consider leftmost error-contributing carry chains contained in the interval  $[j+1,j+2,\dotsc,n]$.  It is enough to compute two quantities $\mu^{+}(j+1)$ and $\mu^{-}(j+1)$, which  stand for the number of input pairs in  positions $j+1,j+2,\dotsc,n$ for which the leftmost error-contributing chain is respectively positive or negative.

It is easy and fast to compute $\mu^{+}(j)$ and $\mu^{-}(j)$ iteratively for $j = n, n-1, \dotsc, 1$, because for going from some $j+1$ to $j$ we only need to consider additional leftmost carry chains, those starting with $j$.  In the computations of the various $\mu^{+}(j)$'s and $\mu^{-}(j)$'s each carry chain needs to considered as leftmost once only: a carry chain starting from $j$ increments $\mu^{+}(j)$ if its error is positive and it increments $\mu^{-}(j)$ if its error is negative.  In consequence, the full algorithm operates in the number of steps stated.

We have not discussed how to determine $\errorchain(\chain)$'s.  Knowing them is of particular interest when a designer needs to decide how to allocate the delays to the various components of the pseudo-adder to optimize the desired performance metrics while staying within some resources budget, such as energy.  Although computing $\errorchain(\chain)$'s is not a part of mathematical modeling, we briefly discuss it.  What one presumably needs to do is to examine how the carries of \chain actually propagate in a specific pseudo-adder under consideration.  The actual procedure depends on the design, the testing, and the computational tools available.

An error of the carry chain is by Definition~\ref{def:3} just an error for any one out of a well-defined set of input pairs. So a straightforward way to determine $\errorchain(\chain)$ is to rely on that definition, picking a convenient pair of inputs generating \chain, and examining the behavior of the pseudo-adder for their addition.  As treating a single \chain is likely to involve effort growing with $n$, in our asymptotic analysis characterized by the number of steps, the work required for one step is not constant, but still likely $\mathrm{O}(n)$.

\subsection{Formulation using probabilities}
\label{sec:form-using-prob}

  It is possible to rewrite \eqref{eq:1} using probabilities, as was done in~\oldpaper.  By overloading our notation, we will write $\errorchain(i,j)$ for the error contributed by the \ijcarrychain.  Let
  \begin{equation}
    \label{eq:10}
    p_{ij}^{\sigma} =\frac{1}{2^{2n}} \nu^{\sigma}(\chain)\quad \text{for} \quad\sigma \in \{+1,-1\}.
  \end{equation}
  Then
  \begin{equation}
    \label{eq:41}
      \exavgroman =  \smashoperator[lr]{ \sum_{1 \leq i \leq j \leq n} }  \; ( p_{ij}^{+} - p_{ij}^{-})  \,  \errorchain(i,j).
    \end{equation}

Formulation using probabilities allows convenient handling of input distribution other than the uniform one.  We do not pursue this further here.

\section{Extension to other statistical measures}
\label{sec:extens-other-stat}

\subsection{Mean squared error}
\label{sec:mean-squared-error}

As before, let $\nu(\chain)$ denote the number of input pairs $(a,b)$ generating \chain and let $\nu(\chain_{1},\chain_{2})$ for $\chain_{1} \not = \chain_{2}$ denote the number of input pairs $(a,b)$ generating both $\chain_{1}$ and $\chain_{2}$.

\begin{theorem}
  \label{thm:4}
The mean squared error, the expectation of $\big( \ex (a,b) \big)^{2}$, is
\begin{equation}
  \label{eq:14}
  \smashoperator[r]{
      \sum_{\chain }
}\;
\nu(\chain) \cdot   \big ( \errorchain(\chain)  \big )^{2}
+
  \smashoperator[r]{
      \sum_{ \chain_{1} \not = \chain_{2}}
}
\nu(\chain_{1},\chain_{2}) \cdot   \errorchain(\chain_{1})  \cdot   \errorchain(\chain_{2})
\end{equation}
divided by $2^{2n}$.
\end{theorem}
% For proof please see Appendix~\ref{sec:proof-theor-refthm:4}
\begin{proof}
  The sum of squared errors is
\begin{align}
  \label{eq:24}
&  \sum_{a,b}   \Biggl (\,
  \smashoperator[r]{
      \sum_{\chain \in \chain(a,b)}
}\,
\errorchain(\chain)  \Biggr )^{2},\\
&\text{which can be written as}\\
&\sum_{a,b}
  \smashoperator[r]{
      \sum_{\chain \in \chain(a,b)}
}\;
\big ( \errorchain(\chain)  \big )^{2}
+ \sum_{a,b}
  \smashoperator[r]{
      \sum_{\chain_{1}, \chain_{2} \in \chain(a,b) \land \chain_{1} \not = \chain_{2}}
}\,
\errorchain(\chain_{1})  \cdot  \errorchain(\chain_{2}),  \\
&\text{and as}\\
& \smashoperator[r]{
      \sum_{\chain }
}\;
\nu(\chain) \cdot  \big ( \errorchain(\chain)  \big )^{2}
+
  \smashoperator[r]{
      \sum_{ \chain_{1} \not = \chain_{2}}
}\,
\nu(\chain_{1},\chain_{2}) \cdot  \errorchain(\chain_{1})  \cdot  \errorchain(\chain_{2}),
\end{align}
from which the theorem follows.
\end{proof}

\subsection{Maximum absolute error}
\label{sec:maxim-absol-error}

This relies on an algorithmic transformation.  We reduce the computation of the maximum absolute value of the error to a simple and well-studied  graph problem.  Construct  a weighted DAG $G=(V,E)$ to specify how carry chains interact, as follows.
\begin{enumerate}
    \item
  The set of vertices $V$ is the set of all pairs $(i,j)$ for $1 \leq i \leq j \leq n$.  Vertex $(i,j)$ corresponds to \ijcarrychain.
    \item The set of edges $E$ is the set of all pairs $\big ( (i_{1},j_{1}),(i_{2},j_{2}) \big)$ for $j_{1} < i_{2}$.  Such an edge indicates that the \ijcarrychainone and the \ijcarrychaintwo do not overlap and therefore they can both be generated by an input pair.
    \item
  The weight $w\big((i,j)   \big)$ of  vertex $(i,j)$ is the error contributed by the \ijcarrychain.  The weight can be negative, positive, or zero.
\end{enumerate}
The weight of a path in $G$ is the sum of the weights of the vertices in it.  First a simple example.

\begin{example}
  \label{ex:4}
  Consider some $16$-bit adder.  The set of the possible chains is $\{\ijcarrychain \mid 1 \leq i \leq j \leq 16 \}$, with the corresponding set of vertices $\bigl \{ (i,j) \mid 1 \leq i \leq j \leq 16 \bigr\}$.  Since no pair of inputs can generate both the \xycarrychain{4}{8} and the \xycarrychain{7}{10}, we do not have edge $\bigl( (4,8), (7,10) \bigr)$.  Since it is possible to have both the \xycarrychain{4}{8} and the \xycarrychain{9}{10} we do have edge $\bigl( (4,8), (9,10) \bigr)$.

  As some pair of inputs can generate the \xycarrychain{4}{8}, the \xycarrychain{9}{10}, and the \xycarrychain{12}{14}, there is a path $(4,8) \rightarrow (9,10) \rightarrow (12,14)$.  And if a pair of inputs generates exactly those three chains, the error generated by this pair will be the weight of the path.  This completes the example.
\end{example}

By formalizing Example~\ref{ex:4} we formulate
\begin{lemma}
  \label{lem:2}
  There is a one-to-one correspondence between the paths in $G$ and the sets of non-overlapping carry chains.
\end{lemma}
% For proof please see Appendix~\ref{sec:proof-lemma-reflem:2}.
\begin{proof}
    A set $\sigma$ of carry chains can appear in a sum if and only if no two carry chains in it overlap.  We show that there is a bijection between the set of all such $\sigma$'s and the set of all the paths in $G$.

  Let $\pi$ be any path, say $v_{1},v_{2},\dotsc,v_{m}$, written more explicitly as $(i_{1},j_{1}),(i_{2},j_{2}),\dotsc,(i_{m},j_{m})$. As there must be an edge between any two consecutive vertices in the path it follows from the definition of $E$ that $i_{1} \leq j_{1} < i_{2}\leq j_{2} < i_{3}\dotsm j_{m-1} < i_{m} \leq j_{m}$. Therefore, for any $1 \leq k < \ell \leq m$, $j_{k} < i_{\ell}$ holds, and therefore the carry chains corresponding to two distinct vertices in $\pi$ do not overlap.

  Let $\sigma$ be any set of non-overlapping carry chains. Sort it to get a sequence $(i_{1},j_{1}), (i_{2},j_{2}), \dotsc,(i_{m},j_{m})$, such that $i_{1}\leq i_{2} \dotsm \leq i_{m}$. Because no two consecutive carry chains in the sequence overlap, all the $i_{k}$'s are distinct and we can assume that $i_{1} < i_{2} < \dotsm < i_{m}$. Again, for the same reason, for every $k$, $1\leq k <m$, the condition $j_{k} <i_{k+1}$ holds. So $\bigl((i_{k},j_{k}),(i_{k+1},j_{k+1})\bigr) \in E$. Therefore the sequence of the carry chains is in fact a path in $G$.
\end{proof}

Let $w_{\min}$ be the minimum weight of the paths and $w_{\max}$ the maximum weight of the paths.  Then
\begin{theorem}
  \label{thm:1}
  The maximum absolute value of an error is $\max\{-w_{\min},w_{\max}\}$.
\end{theorem}
\begin{proof}
 By Lemma~\ref{lem:2}, $w_{\min}$ and $w_{\max}$ are respectively equal to the minimum and the maximum error. Note that we needed to consider $-w_{\min}$ as there may be negative errors with large absolute values.
\end{proof}

Both $w_{\min}$ and $w_{\max}$ can be computed fast by adapting standard all-pairs shortest path algorithms \cite{Cormen 2009}.  Note also that even though $G$ has $\Theta(n^{2})$ vertices, any path in it is of length $\mathrm{O}(n)$ only, which makes computations faster than for general DAGs.

\begin{figure}
  \centering
  \includegraphics[scale=0.85]{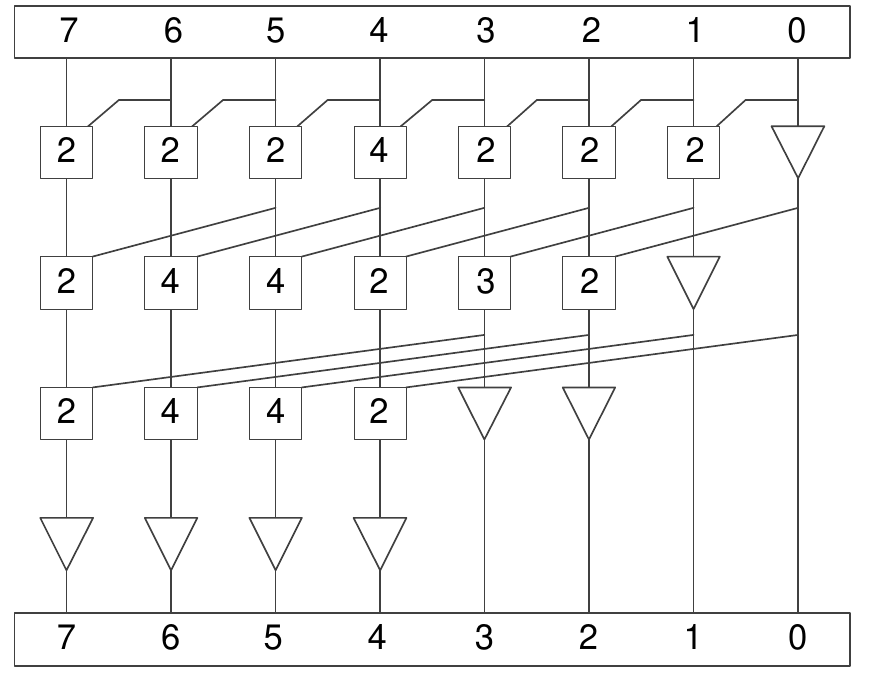}
  \caption{An 8-bit Kogge-Stone adder.  The delays of the \units are written inside the corresponding nodes.}
  \label{fig:Kogge-Stone-adder}
\end{figure}

\begin{figure*}
\centering
% \begin{tabular}{@{\hskip 18pt}l@{\hskip 24pt}l@{\hskip 18pt}}
\begin{tabular}{r@{\hskip 5em}rrrrrrrrr@{\hskip 4em}r@{\hskip 2em}rr}
\toprule
% \multicolumn{1}{c}{\textbf{Position}}&8&7&6&5&4&3&2&1&0\\
&\multicolumn{8}{c}{\hspace{1em}\textbf{Position}}\\
&&7&6&5&4&3&2&1&0\\
\noalign{\vskip 1ex}
\multicolumn{1}{l}{$\boldsymbol{a}\hspace{2em}\longrightarrow$}&$0$&$0$&$1$&$0$&$1$&$0$&$1$&$1$&$0$\\
\multicolumn{1}{l}{$\boldsymbol{b}\hspace{2em}\longrightarrow$}&$0$&$0$&$0$&$1$&$1$&$1$&$0$&$1$&$1$\\
\noalign{\vskip 1ex}
\midrule
\noalign{\vskip 1ex}
% \textbf{Time}&&&&&&&&&&\multicolumn{1}{c}{$s'$}&\multicolumn{1}{c}{$s-s'$}&\\
% \multicolumn{1}{c}{\textbf{Time}}&&&&&&&&&&$s'$&$s-s'$&\\
\textbf{Time}&&&&&&&&&&$s'$&$s-s'$&\\
\noalign{\vskip 1ex}
0&0&\cellcolor{orange}0&\cellcolor{orange}0&\cellcolor{orange}0&\cellcolor{yellow}0&\cellcolor{yellow}0&\cellcolor{yellow}0&0&0&0&+145&\\
1&0&\cellcolor{orange}0&\cellcolor{orange}1&\cellcolor{orange}1&\cellcolor{yellow}0&\cellcolor{yellow}1&\cellcolor{yellow}1&0&1&109&$+36$&\\
2&0&\cellcolor{orange}0&\cellcolor{orange}1&\cellcolor{orange}1&\cellcolor{yellow}0&\cellcolor{yellow}1&\cellcolor{yellow}1&0&1&109&$+36$&\\
3&0&\cellcolor{orange}0&\cellcolor{orange}1&\cellcolor{orange}1&\cellcolor{yellow}0&\cellcolor{yellow}1&\cellcolor{yellow}1&0&1&109&$+36$&\\
4&0&\cellcolor{orange}0&\cellcolor{orange}1&\cellcolor{orange}1&\cellcolor{yellow}0&\cellcolor{yellow}1&\cellcolor{yellow}0&0&1&105&$+40$&\\
5&0&\cellcolor{orange}0&\cellcolor{orange}1&\cellcolor{orange}1&\cellcolor{yellow}0&\cellcolor{yellow}0&\cellcolor{yellow}0&0&1&97&$+48$&\\
6&0&\cellcolor{orange}0&\cellcolor{orange}1&\cellcolor{orange}1&\cellcolor{yellow}0&\cellcolor{yellow}0&\cellcolor{yellow}0&0&1&97&$+48$&\\
7&0&\cellcolor{orange}1&\cellcolor{orange}1&\cellcolor{orange}1&\cellcolor{yellow}0&\cellcolor{yellow}0&\cellcolor{yellow}0&0&1&225&$-80$&\\
8&0&\cellcolor{orange}1&\cellcolor{orange}1&\cellcolor{orange}1&1&0&0&0&1&241&$-96$&\\
9&0&\cellcolor{orange}1&\cellcolor{orange}1&\cellcolor{orange}1&1&0&0&0&1&241&$-96$&\\
10&0&1&0&0&1&0&0&0&1&145&0&\\
\noalign{\vskip 1ex}
\bottomrule
\end{tabular}
\caption{An example of addition in an overclocked Kogge-Stone adder.  Carry chains are highlighted at the instances of time in which they contribute errors.}
\label{table:KSA-errors}
\end{figure*}

\section{Example: A Kogge-Stone-based pseudo-adder}
\label{sec:example:-kogge-stone}

We will now look at a small pseudo-adder based on the KSA and show for it how the signs of the carry chains interact in computing the absolute value of the error produced by sets of carry chains for a pair of inputs, elaborating on Example~\ref{ex:1}.  The KSA is parallel-prefix, carry-lookahead adder.  Ignoring for now the numbers in the vertices, an example of an  $8$-bit KSA is depicted Fig.~\ref{fig:Kogge-Stone-adder}, which we adapt from Weste and Harris~\cite{harris:vlsi2011}, where also a more detailed description of such figures is given.

For us what is important is that the KSA operates with some parallelism.  Therefore, given some allocation of propagation delays, it may happen that both $c_{k-1}=1$ and $c_{k}=1$ but $c_{k}$ propagated to its final destination faster than $c_{k-1}$ did to its final destination.  We note here that irrespective of the allocation of propagation delays in the RCA, that condition can never happen in its simple architecture, and that is why its mathematical modeling could be so much simpler.

%\begin{example}
%  \label{ex:2}
  The numbers in the vertices representing modules in Fig.~\ref{fig:Kogge-Stone-adder} are the delays introduced in those modules.  In Fig.~\ref{table:KSA-errors}, we flesh out Example~\ref{ex:1} for this adder.  We label the two carry chains, $C_{1}$ for the \xycarrychain{2}{4} and $C_{2}$ for the \xycarrychain{5}{7}.   The correct sum is obtained at time $t=10$.  We next consider the case where the output is read at the time $T=7$.

Based on the topology of the KSA  and on the delays in Fig.~\ref{fig:Kogge-Stone-adder}, we compute the values of $\boldsymbol{s'}$ at various instances of time.  Note that sometimes the total error is positive and sometimes negative.  At time $T$, both of the carry chains contribute errors, and we compute the total absolute error produced at that time.
\begin{enumerate}
  \item
$C_{1}$ is an error-contributing carry chain but not the leftmost one, and it contributes an error of $+16$.
  \item
$C_{2}$ is the leftmost error-contributing carry chain.  That is, it is $\chainleft(a,b)$, and it contributes an error of $-96$, so
  \begin{equation}
    \label{eq:5}
\signhat(a,b)=      \sgn \Big(\errorchain \big(\chainleft(a,b) \big) \Big)=-1.
\end{equation}
\end{enumerate}
Therefore, at time $T=7$,\, $\abs{s-s'} = (-1)(-96)+(-1)(+16)=+80$.

Despite its correct hardware, the adder was incorrect in this case because it was overclocked.  A simple modification of this example produced by changing the topology of the adder, would produce a pseudo-adder in which the outputs for some inputs never become correct.

\section{Conclusions}

The method we have presented can be applied to a wide class of circuits implementing approximate adders.  As shown in~\oldpaper for the restricted case of ripple-carry adders, it is now possible to evaluate alternative schemes of allocation of voltages to the adder's components to optimize the quality of its outputs under some statistical measure.  Returning to Fig.~\ref{fig:Kogge-Stone-adder}, the delays in the modules are a function of the voltages applied, which in turn determine the energy consumption.  A key component of an optimization procedure is the ability to generate a good estimate of errors, and our method provides the mathematical means enabling that.

\section*{Acknowledgment}

Some of the results presented here appeared first in a preliminary form in~\cite{KM201303}.

% \ifCLASSOPTIONcaptionsoff
%   \newpage
% \fi



%\bibliographystyle{IEEEtran}
\bibliographystyle{IEEEtran}
\bibliography{references}

\end{document}